\documentclass[letterpaper, 10 pt, conference]{ieeeconf}  

\IEEEoverridecommandlockouts                              
\usepackage{amsmath,amsfonts}
\usepackage{algorithmic}
\usepackage{algorithm}
\usepackage{array}
\usepackage[caption=false]{subfig}
\usepackage{textcomp}
\usepackage{graphicx}
\usepackage{stfloats}
\usepackage{url}
\usepackage{verbatim}
\usepackage{cite}
\usepackage{helvet}
 
\usepackage{bm} 

\usepackage{amsthm} 

\usepackage{mathtools}
\usepackage{amssymb}
\usepackage{commath}
\usepackage{xcolor} 

\usepackage[colorlinks=true,linkcolor=black]{hyperref}
\usepackage{balance}
\usepackage[font={small}]{caption}
\usepackage{siunitx}

\let\labelindent\relax
\usepackage{enumitem} 

\newtheorem{proposition}{Proposition}
\newtheorem{theorem}{Theorem}
\newtheorem{lemma}{Lemma}
\newtheorem{remark}{Remark}
\newtheorem{definition}{Definition}
\newtheorem{assumption}{Assumption}
\newtheorem{corollary}{Corollary}
\newtheorem{example}{Example}

\DeclareMathOperator*{\argmin}{argmin} 

\begin{document}

\title{\LARGE \bf Polytope Volume Monitoring Problem: Formulation and Solution via 
  Parametric Linear Program Based Control Barrier Function} 
 \author{Shizhen Wu, Jinyang Dong, Xu Fang, Ning Sun, and Yongchun Fang$^*$ 
\thanks{This work was done when S. Wu was visiting and X. Fang was working at Nanyang Technological University. ($^*$Corresponding author: Yongchun Fang)}
\thanks{S. Wu, J. Dong, N. Sun and Y. Fang are with the Institute of Robotics and Automatic Information
  System, College of Artificial Intelligence, Nankai University, Tianjin 300353,
  China, and also with the Institute of Intelligence Technology and Robotic
  Systems, Shenzhen Research Institute of Nankai University, Shenzhen 518083,
  China (E-mail:\{szwu; dongjinyang\}@mail.nankai.edu.cn; \{sunn; fangyc\}@nankai.edu.cn).}
\thanks{X. Fang is with  Key Laboratory of Intelligent Control and Optimization for Industrial Equipment of Ministry of Education, Dalian University of Technology, Dalian, 116024, China (E-mail: xufang@dlut.edu.cn).}
}

\maketitle
\thispagestyle{empty}
\pagestyle{empty}

\begin{abstract}
 Motivated by the latest research on feasible space monitoring of multiple control barrier functions (CBFs) as well as polytopic collision avoidance, this paper studies the Polytope Volume Monitoring  (PVM) problem, whose goal is to design a control law for inputs of nonlinear systems to prevent the volume of some state-dependent polytope from decreasing to zero. Recent studies have explored the idea of applying Chebyshev ball method in optimization theory to solve the case study of PVM; however, the underlying difficulties caused by nonsmoothness have not been addressed. This paper continues the study on this topic, where our main contribution is to establish the relationship between nonsmooth CBF and parametric optimization theory through directional derivatives for the first time, to solve PVM problems more conveniently. In detail, inspired by Chebyshev ball approach, a parametric linear program (PLP) based nonsmooth barrier function candidate is established for PVM, and then, sufficient conditions for it to be a nonsmooth CBF are proposed, based on which a quadratic program (QP) based safety filter with guaranteed feasibility is proposed to address PVM problems. Finally, a numerical simulation example is given to show the efficiency of the proposed safety filter. 
  \end{abstract}
  

\section{Introduction}\label{sec:introduction}
 The control barrier function (CBF) based control design method is popular for safety-critical systems\cite{ames2019control}.  CBF-based controllers are typically implemented  pointwisely in the
form of quadratic program (QP), where CBF conditions and input constraints are incorporated as linear inequalities. Due to its advantages of low computational complexity and formal safety guarantees, CBF methods are successfully 
 applied in various autonomous robotic systems \cite{wabersich2023data, garg2024advances, shi2025distributed
 }. 
    
The appearance of  multiple CBFs  and input constraints would impose great challenges in the feasibility  of the resulting QP-based control laws, which has been paid much attention recently \cite{glotfelter2017nonsmooth, xu2018constrained,  xiao2022sufficient,  isaly2024feasibility, parwana2023feasible}. Some primary works study sufficient conditions of the compatibility of  multiple CBFs \cite{xu2018constrained, xiao2022sufficient}. More recently, in \cite{isaly2024feasibility}, the nonsmooth CBF, dating backing to \cite{glotfelter2017nonsmooth}, is introduced as a systematic tool to address challenges in guaranteeing the feasibility and continuity. Meanwhile, a simpler way of addressing the feasibility problem appears in \cite{parwana2023feasible}. In detail,  ``\emph{Feasible CBF Space}'' is defined to describe the state-dependant polytope formulated by CBF conditions and input constraints. Then, ``\emph{Feasible Space Monitoring}'' (FSM) problem is established to monitor the feasibility of the resulting QP-based control laws, for which Chebyshev ball and inscribing ellipsoid methods are proposed to build parametric convex program based CBFs.  For engineering implementations, the differentiable optimization toolbox is applied  to approximate the gradient of the resulting  (nonsmooth implicit) barrier function candidate by some heuristic gradient values. As concluded in \cite{parwana2023feasible}, the nonsmooth control theoretic considerations are left for future work. To the best of our knowledge, this problem has been open until now.  Such observation motivates this paper. 
  
Meanwhile, researchers have been gradually interested in building CBFs via parametric optimization-defined metrics for other problems, such as network connectivity maintenance \cite{ong2023nonsmooth} and  polytopic collision avoidance (PCA) \cite{dai2023safe, wei2024collision, thirugnanam2025control,    wu2025optimization,  thirugnanam2022duality, thirugnanam2023nonsmooth}. However, it is difficult to extend these results to solve the above open problem in FSM. Roughly speaking, there are three potential ways, yet all meet significant difficulties: 
1) In \cite{ong2023nonsmooth}, the Laplacian's eigenvalue is chosen as a barrier function candidate, then, a nonsmooth CBF condition in the generalized gradient form is formulated. 
The idea of building nonsmooth CBF conditions based on generalized gradient can only be applied to specific problems, and it is not applicable to other parametric optimization problems.  
2) In \cite{thirugnanam2022duality, thirugnanam2023nonsmooth}, the minimum distance field for PCA, i.e., a parametric quadratic program defined metric, is used to formulate a nonsmooth CBF condition based on the directional derivative of the KKT solution. How to extend it to other problems is also not obvious. 
3) Recently, the nonsmoothness in PCA  has been avoided  by circumscribing polytopes  into strictly convex shapes to construct differentiable optimization defined CBFs \cite{dai2023safe, wei2024collision, thirugnanam2025control} or by proposing some nonconservative approximation of signed distance field to establish the optimization-free smooth CBF \cite{wu2025optimization}. 
This indirect approach avoids the challenges of nonsmooth parametric optimization, transforming the original simple LP into a nonlinear program. 
In summary, the existing works about parametric optimization defined CBF methods \cite{ong2023nonsmooth, dai2023safe, wei2024collision, thirugnanam2025control,    wu2025optimization,  thirugnanam2022duality, thirugnanam2023nonsmooth} could not be directly applied to address the difficulties of nonsmoothness in FSM problems. 
 
This paper extends the study of Feasible Space Monitoring \cite{parwana2023feasible} and addresses the unresolved difficulties in nonsmooth parametric optimization based CBFs  \cite{ong2023nonsmooth, dai2023safe, wei2024collision, thirugnanam2025control,  wu2025optimization,  thirugnanam2022duality, thirugnanam2023nonsmooth}. First, a broader class of control problems is formulated, called Polytope Volume Monitoring  (PVM). By identifying this kind of problems, 
FSM problem \cite{parwana2023feasible}  and Polytopic Intersectability Maintenance, i.e., the opposite problem of PCA \cite{dai2023safe, wei2024collision, thirugnanam2025control, wu2025optimization,  thirugnanam2022duality}, can be regarded as different special examples of PVM. 
This generalization enhances the applicability of the proposed method to a wider range of problems.
Second, to address difficulties in analysis and synthesis caused by inherent nonsmoothness, this paper attempts to establish a bridge between the directional-derivative-defined nonsmooth CBF \cite{wiltz2023construction}  and  parametric optimization theory \cite{still2018lectures}. 
In detail, our main contributions can be summarized as follows: 
\begin{itemize} 
\item [1)] Compared with 
existing works \cite{ong2023nonsmooth, dai2023safe, wei2024collision, thirugnanam2025control,  wu2025optimization,  thirugnanam2022duality, thirugnanam2023nonsmooth}, 
this paper gives a direct way to analyze and synthesize parametric optimization defined nonsmooth CBFs
by merging \cite{wiltz2023construction} and \cite{still2018lectures} into a unified framework for the first time.
\item [2)] Within this unified framework, a parametric linear program based nonsmooth barrier function candidate is established for PVM, based on which sufficient conditions of being a nonsmooth
CBF are proposed. Then a a feasibility-guaranteed  QP-based safety filter is proposed to solve the PVM problem. The above nonsmooth analysis and considerations are not addressed in \cite{parwana2023feasible}.
\item [3)] The efficiency of the  proposed safety filter is validated in a numerical simulation where a vehicle is driven to perform safe navigation, i.e.,  a reach-avoid task under multiple obstacles
and input constraints.
\end{itemize} 
 
The paper is organized as follows. The CBF method is briefly reviewed in Section \ref{sec:background}, and problem formulation and preliminary analysis are given in Section \ref{sec:problem}. Section \ref{sec:synthesizing} reports the main results. The referred nonsmooth parametric optimization theory is included in Appendix. 

\emph{Notations:} Throughout this paper, $\operatorname{col}(x_1,x_2):=[x_1^{\top},x_2^{\top}]^{\top}$ for any two column vectors $x_1, x_2$. For any vector $x$, $\|x\|$ denotes its Euclidean norm. $|\mathcal{J} |$ denotes the cardinality of  any finite number set $\mathcal{J}$. $0$ and $\emptyset$ denote the zero vector and empty set respectively in suitable dimension. 
The interior and boundary of a set $\mathcal{C}$ are indicated by $\operatorname{Int}(\mathcal{C})$ and $\partial \mathcal{C}$. A continuous function $\alpha: \mathbb{R} \rightarrow \mathbb{R}$ 
is called in extended class-$\mathcal{K}$ if it is strictly increasing and with $\alpha(0) = 0$ \cite{ames2019control}. For any function $f:\mathbb{R}^{n} \rightarrow \mathbb{R}$, if exists, its generalized gradient at $x$ is denoted by $\partial f(x)$, and  $ f' (x; d)$ denotes  its directional derivative at $x$ along direction $d$. For brevity, its partial derivative $\frac{\partial f}{\partial x}(x)$,  expressed as a row vector,  is denoted by $\partial_{x} f$ sometimes. 

\section{Preliminaries: Control Barrier Function}\label{sec:background}

Consider a control-affine nonliear system
\begin{equation} \label{eq:N-affine-systems}
\dot{x}(t) =f_{0}(x(t),u(t)):=   f(x(t)) + g(x(t))u(t),  
\end{equation}
where $x\in \mathbb{R}^{n} $ is state,  $u\in \mathcal{D}_{u}\subset \mathbb{R}^{m} $ is control input,  and $f,g$  are Lipschitz continuous. Next, the nonsmooth barrier function (NBF) candidate   \cite{glotfelter2017nonsmooth}  is  introduced. 

\begin{definition}\label{def:NBFcandidate}
Given a nonempty closed set $\mathcal{C} \subset  \mathbb{R}^n$ with no isolated point, 
a  locally Lipschitz continuous  function $h(\cdot): \mathbb{R}^n \rightarrow \mathbb{R}$ is called a NBF candidate on set $\mathcal{C}$ if 
\begin{equation} \label{eq:a-safe-set}
\begin{aligned}
h(x) &=0,~ \forall x\in \partial \mathcal{C}; \quad h(x) >0,~ \forall x\in \operatorname{Int}(\mathcal{C});  \\  
h(x) &<0,~ \forall x\in \mathbb{R}^n/\mathcal{C}. 
\end{aligned}
\end{equation} 
\end{definition}

 For the demands of this paper,  we 
 only require from now on that $ h $ is locally Lipschitz continuous and  the (one-sided) directional derivative of $h(x)$ at point $x$ along  any  direction $d$ exists, i.e., the following limit is finite and abbreviated as 
$$ 
 h'(x;d) := \lim_{\sigma \rightarrow 0^{+}} \frac{h(x+\sigma d) - h(x) }{\sigma}. 
$$
By doing this, one can relax the differentiability of $ h $ and  define a broader class of CBFs as follows.

\begin{definition}\label{def:ncbf}
A locally Lipschitz continuous function $h: \mathbb{R}^n \rightarrow \mathbb{R}$ is a nonsmooth control barrier function (NCBF) for system \eqref{eq:N-affine-systems} if 
$h $ is a NBF candidate on the  closed set $\mathcal{C} \subset  \mathbb{R}^n$, and there exists 
 an extended class-$\mathcal{K}$ function $\alpha: \mathbb{R} \rightarrow \mathbb{R}$, 
  an open set $\mathcal{D}$ that satisfies  $\mathcal{C} \subset \mathcal{D} \subseteq \mathbb{R}^n$, 
such that:
\begin{align}
  \label{eq:cbf condition}
  \sup_{u\in \mathcal{D}_{u}} \left\{   h' \left(x;  f_0(x,u) \right)  \right\} \geq -\alpha(h(x)), \quad \forall x\in \mathcal{D}. 
\end{align}
\end{definition}

In \cite{wiltz2023construction}, the same limitation value as $ h'(x;f_0(x,u)) $ is used to define NCBF, yet without introducing the concept of directional derivative.
In this paper, the directional derivative defined  NCBF   \eqref{eq:cbf condition}
is introduced to build a bridge between the NCBF theory and the parametric optimization theory. 

\begin{lemma}\cite[Theorem 1]{wiltz2023construction}~\label{lem:NCBF}
  If $h$ is a NCBF for system \eqref{eq:N-affine-systems}, then the closed set $\mathcal{C}$ is forward control-invariant, that is,  
  there exists a measurable and locally bounded function $u(t)\in \mathcal{D}_{u}$, such that  $x_0 \in \mathcal{C}$ implies that $x(t) \in \mathcal{C},~\forall t \in\left[0, T\right]$  for every solution.
\end{lemma}

When $h(x)$ is differentiable, it holds that $h'(x; f_0(x,u))=\frac{\partial h}{\partial x}(x) \left( f(x)+g(x)u  \right)$, which is a linear function of input $u$ at every $x$, and the following  
 is a  QP-defined control law:  
\begin{subequations}\label{eq:uxargmin}
\begin{align} 
u^{*}(x) =\argmin_{u\in \mathcal{D}_{u}}~ & \| u - u_{0} \|^2,
\label{eq:uxargmin1} \\
 \text{ s.t. }  &    h' \left(x;  f_0(x,u) \right)  \geq -\alpha(h(x)), \label{eq:uxargmin2}
\end{align}
\end{subequations}
where $u_{0}$ is any given continuous nominal control law. 
When $h(x)$ is no longer differentiable but the directional derivative  $h'(x; f_0(x,u))$ exists,   
the inequality \eqref{eq:uxargmin2} may no longer be  a linear constraint of input $u$.  

\begin{remark}\label{eq:supuUzfg}
Different from  \eqref{eq:cbf condition}, the more popular NCBF is defined by the set-valued Lie derivative in \cite{glotfelter2017nonsmooth}, \cite{glotfelter2020nonsmooth, ong2023nonsmooth}:
\begin{align}
  \sup_{u\in \mathcal{D}_{u}} \left\{   \min_{\zeta\in \partial h(x) }  \zeta^{\top} f_0(x,u)    \right\} \geq -\alpha(h(x)), \quad \forall x\in \mathcal{D},  
  \label{eq:uinUminpart}
\end{align}
where $\partial h(x)$ denotes the generalized gradient of $h$ at $x$. Other similar NCBFs defined by $\partial h(x)$ can be found in \cite{ghanbarpour2022optimal}. Condition \eqref{eq:uinUminpart} is suitable for the case that  $\partial h(x)$ could be easily calculated or estaimated. The following context will show a typical case that $\partial h(x)$ cannot be obtained obviously, while the directional derivative $h'(x; \cdot)$ is available.  $\hfill\square$
\end{remark}

\begin{figure*}[!htp]
  \centering
  \subfloat[Polytope volume monitoring problem and examples. ]{\includegraphics[width=0.4\textwidth]{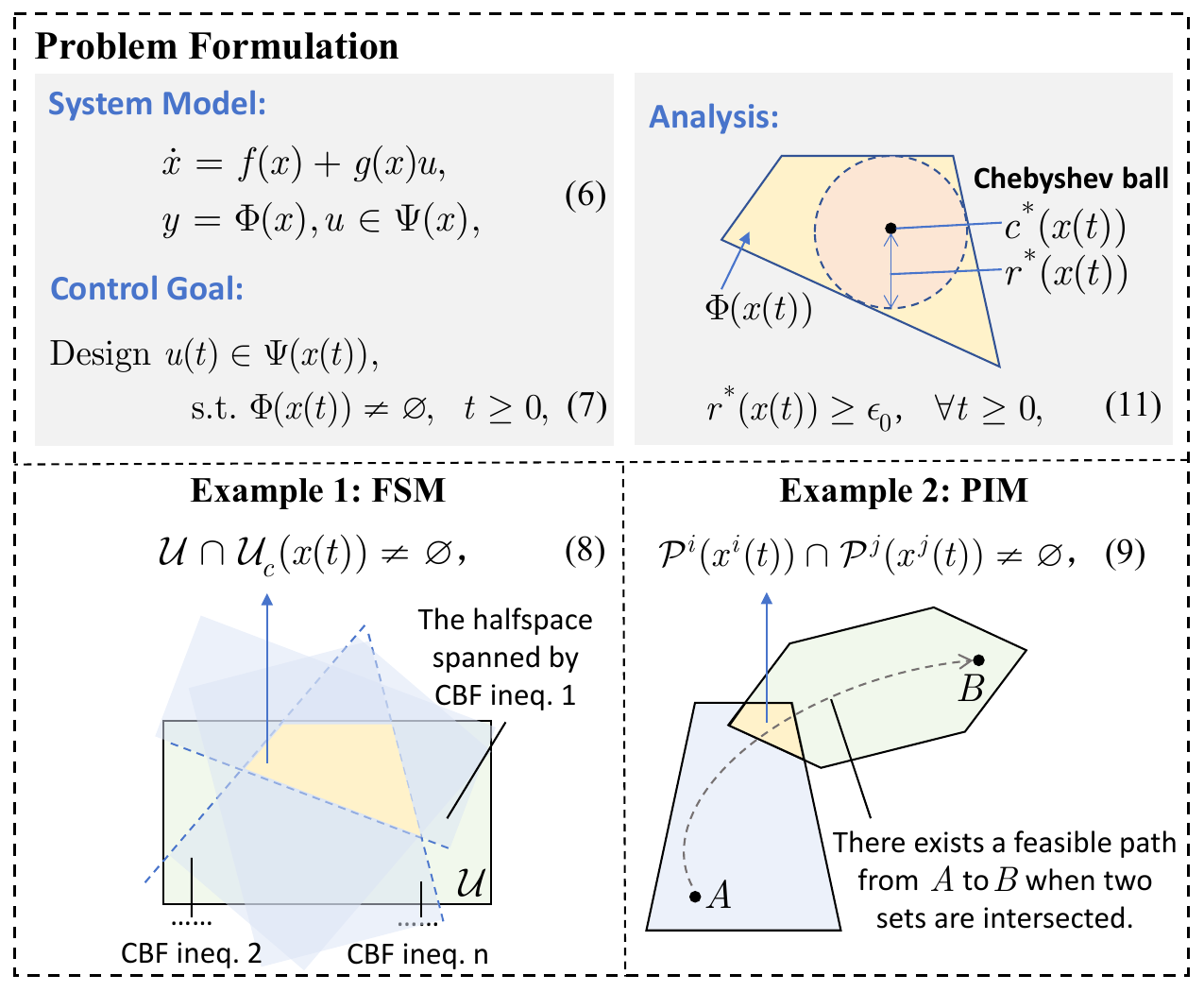} \label{fig:Figure-new}}
  \subfloat[The structure of the proposed safety filter.]{\includegraphics[width=0.588\textwidth]{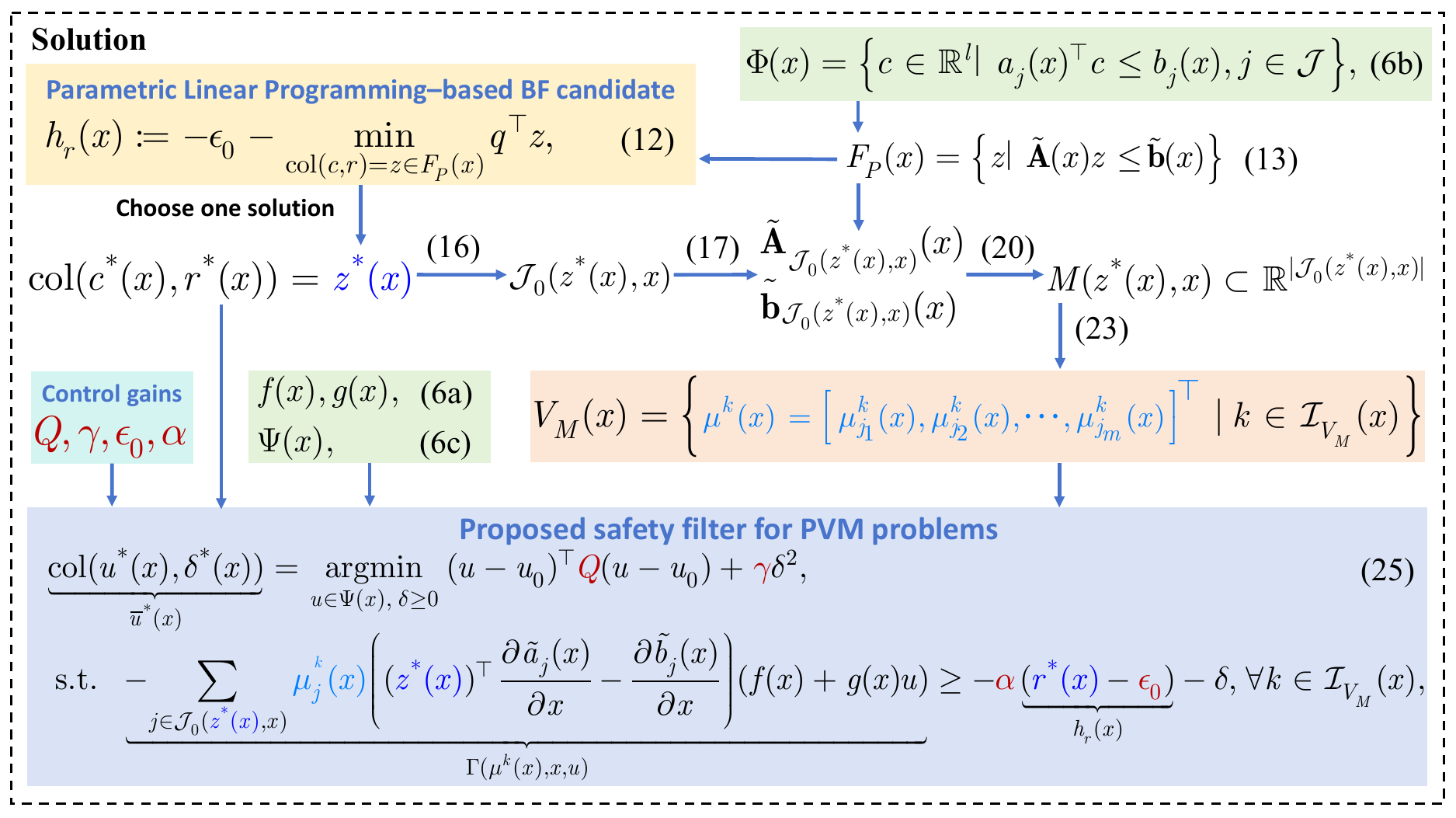} \label{fig:Figure-problem} }
  \caption{A summary of the polytope volume monitoring problem and the parametric linear program based control barrier function.}
  \label{fig:Figure-problem-solution}
    \vspace{-0.6cm}
\end{figure*}

\section{Problem Formulation}\label{sec:problem}
\subsection{Problem Statement}\label{sec:problem-statement}

For simplicity of discussions, consider the control-affine system  \eqref{eq:N-affine-systems} with a  polytopic set-valued output and a polytopic input constraint:
\begin{subequations}\label{eq:N-affine-input-output}
 \begin{align} 
  \dot{x} &=  f(x) + g(x)u, \label{eq:N-affine-systems2}  \\
  y &= \Phi(x) := \left\{ c \in  \mathbb{R}^{l} \mid a_{j}^{\top}(x) c \leq b_{j}(x), j\in \mathcal{J} \right\},   \label{eq:N-affine-output} \\
  u &\in \Psi(x) :=\left\{ u \in \mathbb{R}^{m} \mid    \mathfrak{a}_{\mathfrak{i}}^{\top}(x) u \leq \mathfrak{b}_{\mathfrak{i}}(x), \mathfrak{i}\in \mathfrak{J}  \right\},   \label{eq:N-affine-input} 
 \end{align}  
\end{subequations} 
 where $\mathcal{J}:=\{1,2,...,N\}$ and $\mathfrak{J}:=\{1,2,...,\mathfrak{N} \}$. 
For convenience, the functions appearing in definitions of $\Phi(x)$ and $\Psi(x)$ are assumed to be continuously differentiable, and without loss of generality,  the following assumptions are imposed to develop an implementable controller. 

\begin{assumption}\label{ass:1}
The  polytopic input constraint  set  $\Psi(x)$ is always nonempty and the vectors $a_{j}(x),j\in \mathcal{J}$ are  always nonzero. The state-dependent polytopic set $\Phi(x)$ is uniformly bounded, i.e.,  there exists a compact set $ \mathcal{D}_{y} \subset \mathbb{R}^l $, such that $\Phi(x)\subseteq \mathcal{D}_{y},\forall x\in \mathbb{R}^{n} $. 
\end{assumption}

It is worth noting that  any emptyset is naturally bounded, thus it is not required in Assumption \ref{ass:1}  that $\Phi(x)$ is always nonempty.  Instead, we allow there exist some points $x$ at which $\Phi(x) = \emptyset$.
In fact, this paper will attempt to synthesize a controller to prevent the interior of  $\Phi(x(t))$ from vanishing, as formally stated below.
 

\textbf{Polytope Volume Monitoring Problem:} Consider the system \eqref{eq:N-affine-input-output} with a polytopic set-valued output $\Phi(x)$ and a polytopic input constraint set $\Psi(x)$. The goal is to design a control law for $u$ to prevent  the volume of $\Phi(x(t))$ from decreasing to zero, i.e., to ensure the nonemptness of  the interior of $\Phi(x(t))$: 
\begin{align}
x(t) \in \mathcal{A} := \{ x \mid \operatorname{Int} (\Phi(x)) \neq \emptyset \}, \quad \forall t\geq0. \label{eq:def-nonempty}
\end{align}

There are some typical situations where the  state-dependant polytope $\Phi(x(t))$ might be empty if the input $u$ is designed improperly. 

\begin{example}[Feasible Space Monitoring \cite{parwana2023feasible}]\label{exa1} 
In Fig. \ref{fig:Figure-problem-solution} (a),  given the polytopic input constraint set $\mathcal{U} $, by constructing the set composed of multiple differentiable CBFs $h_i, i\in \mathcal{I}$: 
$\mathcal{U}_{c}(x):=\{ u \mid \frac{\partial h_{i}}{\partial x}(x) \left( f(x)+g(x)u  \right) \geq -\alpha_{i}(h_{i}(x)), i\in \mathcal{I}\}$, 
the FSM problem aims to design $u$ such that
\begin{equation}
  \mathcal{U} \cap \mathcal{U}_{c}(x(t)) \neq \emptyset, ~\forall t\geq0. \label{eq:UUcnonempty}
\end{equation}   
By choosing $\Phi(x)= \Psi(x)= \mathcal{U} \cap \mathcal{U}_{c}(x) $, FSM of any QP-based controllers can be regarded as the case studies of the formulated PVM problem. 
\end{example}
 
In addition to the above examples in input space, PVM problems also appear in state space. 
In many applications, it is necessary to maintain some kind of connectivity of two sets,  for instance, a person standing on vehicle $i$ needs to transfer to vehicle $j$. During this process, $\mathcal{P}^i(x^i)$ and  $\mathcal{P}^j(x^j)$ should be connected, as shown in \eqref{eq:PiPjnonempty}.

\begin{example}[Polytopic Intersectability Maintenance]\label{exa2} 
Consider two standard nonholonomic vehicles $i,j$, for each $\mathfrak{i} \in \{i,j\}$,  
$x^\mathfrak{i}=[{p}_x^\mathfrak{i},{p}_y^\mathfrak{i}, {\theta}^\mathfrak{i}]^{\top}$ 
and $ u^\mathfrak{i} =[ v^\mathfrak{i}, \omega^\mathfrak{i}]^{\top}\in \mathcal{U}^{\mathfrak{i}}$ are state and input vectors respectively, where $\mathcal{U}^{\mathfrak{i}}$ denotes a hyperrectangle.  
Let  $\mathcal{P}^i(x^i), \mathcal{P}^j(x^j)$ denote their physical domain or sensor ranges, as in Fig. \ref{fig:Figure-problem-solution} (a).  The goal of PIM is to maintain the intersection  of two polytopes nonempty: 
\begin{equation} \label{eq:PiPjnonempty}
   \mathcal{P}^i(x^i(t)) \cap \mathcal{P}^j(x^j(t)) \neq \emptyset, ~\forall t\geq0,
\end{equation}   
which is  opposite to the goals of polytopic collision avoidance \cite{dai2023safe, wei2024collision, thirugnanam2025control, wu2025optimization,  thirugnanam2022duality, thirugnanam2023nonsmooth}.
From the perspective of PVM, $\Phi(x)=\mathcal{P}^i(x^i) \cap \mathcal{P}^j(x^j) $ and $\Psi(x)=\mathcal{U}^{i}\times \mathcal{U}^{j}$ in this example.  
\end{example} 

\subsection{Preliminary Analysis}
 
The 
radius $r\geq0$ of the largest sphere contained inside the  polytope $\Phi(x) $ defined in \eqref{eq:N-affine-output}, called
Chebyshev ball \cite[Chapter 8]{boyd2004convex}, can be computed by the following linear
program (LP)  at every $x$: 
\begin{subequations}\label{eq:Chebyshev-ball}
\begin{align}
  r^{*}(x):= &\max_{c\in \mathbb{R}^{l},  r\geq 0} ~r = - \min_{c\in \mathbb{R}^{l},  r\geq 0} ~ -r,  \\
  \text{ s.t. }  & a_j^{\top}(x) c + \|a_j(x)\| r \leq b_j(x),~  \forall j\in \mathcal{J}.
  \end{align}
\end{subequations} 
Given any user-defined threshold value $\epsilon_{0}>0$, as illustrated in Fig. \ref{fig:Figure-problem-solution} (a), one can know that  
 the goal in \eqref{eq:def-nonempty} holds if 
\begin{align}
r^{*}(x(t)) \geq \epsilon_{0} ,~\forall t\geq0. \label{eq:rxtgeqe}   
\end{align}

\begin{remark}
The reference \cite{parwana2023feasible} shows the first study on choosing the feasible space volume as a BF candidate to solve 
feasible space monitoring of CBFs, i.e. a case study of the PVM problem here. However, it 
 focuses on engineering implementations, applying the differentiable optimization toolbox, $\mathrm{cvxpylayers}$ \cite{agrawal2019differentiable}, to approximate the gradient of the nonsmooth implicit function  $ r^{*}(x)$ by some heuristic gradient values. 
It is noted in  \cite[Remark 1]{parwana2023feasible} that a more formal nonsmooth control theoretic analysis and consideration of generalized gradient $\partial r^{*}(x)$, is left for future work. As far as the authors know, this problem has been still  open up to now due to the inherent difficulties in estimating  $\partial r^{*}(x)$. In the following, more in-depth analysis will be performed to design a safety filter \eqref{eq:sf-PVM} without requring  $\partial r^{*}(x)$.
  $\hfill\square$ 
\end{remark}

\section{Main results}\label{sec:synthesizing}

To solve the aforementioned open problem,  the nonsmooth parametric optimization theory \cite{still2018lectures}, as shown in Appendix \ref{sub:pcp},  is introduced to develop a general approach for synthesizing a NCBF-based safety filter to solve PVM problems for the first time. 
Without calculating the generalized gradient $\partial r^{*}(x)$, the NCBF condition is  established based on its directional derivative for convenience.

\subsection{Establishing a Parameteric LP-defined NBF Candidate}
 
Inspired by \eqref{eq:Chebyshev-ball} and \eqref{eq:rxtgeqe}, one can define an implicit function $h_{r}(x)$ via the following  parametric linear program: 
\begin{equation}\label{eq:hxmin}
\begin{aligned}
 h_{r}(x) := &   r^{*}(x) - \epsilon_{0} = - \epsilon_{0} - \!\min_{z\in F_P(x)}  q^{\top} z ,
 \forall x \in \mathcal{A}, 
\end{aligned}
\end{equation}
where $q:=\operatorname{col}{(0_{l \times 1},-1)}$ and $ 
z:=\operatorname{col}{(c,r)} \in \mathbb{R}^{ l+1}$. The feasible set $F_P(x)$ is defined as  
\begin{equation}\label{eq:F_Px}
\begin{aligned}
F_P(x) :=\left\{z\mid  \tilde{\mathbf{A}}(x) z\leq \tilde{\mathbf{b}}(x) \right\}, 
\end{aligned} 
\end{equation} 
where $\tilde{\mathbf{A}}(x) \in \mathbb{R}^{(N+1) \times (l+1)}$ and $\tilde{\mathbf{b}}(x) \in \mathbb{R}^{N+1} $ are  
\begin{align}
\tilde{\mathbf{A}}(x)   :=\left[ \! \! \!   \begin{array}{cc}   
   0_{1 \times l},& -1   \\
  a_1^{\top}(x), & \|a_1(x)\|   \\
  a_2^{\top}(x), & \|a_2(x)\|   \\
 \cdots  & \cdots \\
  a_N^{\top}(x), &  \|a_N(x)\| 
\end{array} \! \! \!  \right] , ~
\tilde{\mathbf{b}}(x)   :=\left[ \! \!  \!  \begin{array}{c}   
   0 \\
  b_1(x) \\
  b_2(x) \\
  \cdots  \\
  b_N(x)
\end{array} \! \!  \!  \right]. \label{eq: Atildebtilde}
\end{align}
For brevity, let $\tilde{a}_j^{\top}(x)$ and $\tilde{b}_j(x)$ denote the $j$-th row of $\tilde{\mathbf{A}}(x)$ and $\tilde{\mathbf{b}}(x)$ respectively, i.e., 
$$
\tilde{a}_j(x) := \tilde{\mathbf{A}}^{\top}(x) e_{j}, ~ 
\tilde{b}_j(x) := \tilde{\mathbf{b}}^{\top}(x) e_{j}, \forall j\in [N+1],
$$  
where  $e_j $ is the $j$-th  column of the $(N+1)$-dimensional identity matrix and $[N+1]:=\{1,2, \cdots, N+1\}$.

Next, the basic properties of the parametric linear program in \eqref{eq:hxmin} are shown first, where Mangasarian Fromovitz Constraint Qualification (MFCQ), as in {Definition} \ref{def:LICQ-MFCQ},  is a key  condition in studying the continuity and directional differentiability of $r^{*}(x)$.
 
 \begin{proposition}\label{theo:pLPhx}
Consider the set $\mathcal{A}$ defined in \eqref{eq:def-nonempty} and the function $h_{r}(x)$ in  \eqref{eq:hxmin}.  
The following conclusions hold: 
\begin{itemize} 
 \item[i)] The interior of $F_P(x)$ is nonempty for all $x\in \mathcal{A}$. 
 \item[ii)] Furthermore, the set $F_P(x)$ and the function $h_{r}(x)$ are uniformly bounded for all $x\in \mathcal{A}$.  The solution map
 \begin{equation} \label{eq:Sxargmin} 
  \begin{aligned}
 S(x) :=& \argmin_{z\in F_P(x)}~  q^{\top} z ,
  \end{aligned}
 \end{equation}
 is on the boundary of $F_P(x)$, i.e., $ S(x)\subset \partial F_P(x)$. 
 \item[iii)] For all $x\in \mathcal{A}$,  the MFCQ holds at any $z \in F_P(x)$.
\end{itemize} 
\end{proposition}

\begin{proof} 
i) First, it is shown that  $F_P(x)\neq \emptyset,\forall x \in \mathcal{A}$. Recalling 
$$ 
\operatorname{Int} (\Phi(x)) = \left\{c \mid  a_j^{\top}(x) c < b_j(x),~  \forall j\in \mathcal{J}  \right\}, 
$$
for $x \in \mathcal{A}$, since $\operatorname{Int} (\Phi(x)) $ is nonempty and uniformly bounded (according to Assumption \ref{ass:1}), one can find a sufficiently small  
$\varepsilon_0(x)>0$ and  a point $c_0(x)\in \mathcal{D}_{y}$, such that 
$$a_j^{\top}(x) c_0(x) + \varepsilon_0(x)  < b_j(x),~  \forall  j\in \mathcal{J}.$$
For every $j $, since $\|a_j(x)\|\neq 0$ as in {Assumption} \ref{ass:1}, 
there exists $c_2(x)\neq 0$, s.t.  
$$\varepsilon_0= a_j^{\top} c_2 =\| a_j\| \underbrace{\|c_2\| \cos \langle a_j, c_2  \rangle}_{=: r_j(x)}  >0.$$
Let $r_0:=\min_{j\in \mathcal{J}}~ r_j$, one has $r_0>0$, $\| a_j\| r_0 \leq \varepsilon_0$, and  
$$a_j^{\top}(x) c_0(x) + \|a_j(x)\| r_0(x) < b_j(x),~  \forall  j\in \mathcal{J},$$
which means that $  \operatorname{Int}(F_P(x)) \neq \emptyset $.

ii) The following shows the uniform boundedness of $F_P(x)$.  Recalling the expression of $F_P(x)$ in \eqref{eq:F_Px}:
$$\begin{aligned}
F_P(x) :=\left\{z\mid r\geq0,  {a}_{j}^{\top} c + \| {a}_{j} \| r \leq  {b}_{j} , \forall j\in \mathcal{J} \right\}, 
\end{aligned}$$ 
from which one knows that 
$$
\begin{aligned}
0 \leq r \leq & \max_{c\in \Phi(x)}  \min_{j\in \mathcal{J}} \frac{  {b}_{j}(x) -  {a}_{j}^{\top} (x) c }{ \|{a}_{j}(x) \|  } \\
=&  \max_{c\in \Phi(x)} d(c, \partial \Phi(x)) \leq \max_{c\in \mathcal{D}_{y}} d(c, \partial \mathcal{D}_{y}) =: r_{d},
\end{aligned}
$$
where $d(c, \mathcal{C} ):= \min_{c'\in  \mathcal{C}} \| c- c' \|$ denotes the distance function from a point $x$ and a set $\mathcal{C}$, the compact set  $\mathcal{D}_{y}$ is derived from Assumption \ref{ass:1} and $r_{d}$ is a positive constant.
Furthermore, one has 
$$
F_P(x) \subset \Phi(x) \times [0, r_{d}]  \subset  \mathcal{D}_{y} \times [0, r_{d}], 
$$
which means that the set $F_P(x)$ is uniformly bounded, and functions $r^*(x)$ and  $h_r(x)$ are also uniformly  bounded. 
Due to the fact that \eqref{eq:hxmin} is a LP problem, the optimization value is obtained at a point $z$ only if $z$ is on the boundary $\partial F_P(x)$. 

iii) The MFCQ condition holds for the optimization problem in \eqref{eq:hxmin} since it is a LP piecewise at $x$.  
Recalling the fact in \cite{solodov2010constraint} that for $\mathcal{C}^{1}$ convex programs with affine inequalities, the MFCQ condition is equivalent to the Slater condition.  
Since  we have proved in {Theorem} \ref{theo:pLPhx} i)  that  $  \operatorname{Int}(F_P(x)) \neq \emptyset, \forall x\in \mathcal{A}$, there exists a point $z\in \operatorname{Int}(F_P(x))$ such that the inequalities in \eqref{eq:F_Px} holds strictly, hence the Slater condition (as well as the MFCQ condition) naturally holds at any $z \in F_P(x)$.  
\end{proof}
 
Based on {Proposition \ref{theo:pLPhx}}, one can prove that $h_{r}(x)$ is a NBF candidate by verifying the conditions in Definition \ref{def:NBFcandidate}. 
 
\begin{theorem} \label{lem:the-function-sd-candidate-NCBF}
The  Parametric LP defined implicit function $h_{r}(x)$ in \eqref{eq:hxmin}   is locally Lipschitz continuous on $\mathcal{A}$, and it is a  NBF candidate on the closed set $\mathcal{S}:=\{ x\mid h_r(x) \geq 0  \}\subset \mathcal{A}$. 
\end{theorem} 

\begin{proof}
Recalling Proposition \ref{theo:pLPhx} ii) and iii) that  ${F}_P(x)$ is uniformly bounded for $x\in \mathcal{A}$ and  the MFCQ condition holds, the compact set $C_0$ required in Lemma \ref{theorem:PCP} exists, hence the conditions in Lemma \ref{theorem:PCP} are satisfied, thus $h_{r}(x)$ is  locally Lipschitz continuous.  
Since $\mathcal{S}$ is defined by $h_{r}(x)$, $h_{r}(x)$ is a NBF candidate on the set $\mathcal{S}$ according to Definition \ref{def:NBFcandidate}. 
\end{proof}
 
 
\subsection{Building  the NCBF Condition}\label{sec:cbNBFc}

This section aims to establish the NCBF condition based on existing results about the directional differentiability of parametric convex programs in \emph{Lemma \ref{theorem:PCP2}}. 

First, according to the feasible set $F_P(\cdot)$ in \eqref{eq:F_Px}, the active index set\footnote{Inspired by \cite{glotfelter2020nonsmooth}, $ \mathcal{J}_0( {z},x)$ can be replaced by the socalled \emph{almost-active} index set 
$  \mathcal{J}_{\varepsilon}( {z},x):=\{j \in \mathcal{J} \mid  \| \tilde{a}_j^{\top}(x) {z} - \tilde{b}_j(x)\| \leq \varepsilon  \}$ for real implementation, where $\varepsilon$ is a sufficiently small scalar, such as  $ \varepsilon=10^{-4}$.} is defined for every $ {x}$ and feasible $ {z} \in F_{P}( {x})$ as:
\begin{equation}\label{eq:J0zleft}
  \begin{aligned}
 \mathcal{J}_0( {z},x) 
:= & \left\{j \in [N+1]  \mid  \tilde{a}_j^{\top}(x) {z} = \tilde{b}_j(x) \right\}.   
  \end{aligned}
\end{equation}
For every $\bar{z}, \bar{x}$ satisfying $\bar{z}\in S(\bar{x})$, recalling $ S(\bar{x}) \subset  \partial F_P(\bar{x})$ in {Proposition} \ref{theo:pLPhx} iii), it holds that 
$$ \mathcal{J}_0(\bar{z}, \bar{x}) \neq \emptyset , \text{ i.e., } | \mathcal{J}_0(\bar{z}, \bar{x})|\geq 1. $$
Based on \eqref{eq: Atildebtilde} and \eqref{eq:J0zleft}, one can further define the active matrix of $\tilde{\mathbf{A}}(x)$ and  the active  vector of $\tilde{\mathbf{b}}(x)$  for every $x$: 
\begin{equation}\label{eq:J0zleft2}
  \begin{aligned}
  \tilde{\mathbf{A}}_{ \mathcal{J}_0(\bar{z}, \bar{x})} (x) 
  :=  \! \! \left[\!\! \! \begin{array}{c}
    \tilde{a}_{j_1}^{\top}(x) \\
    \tilde{a}_{j_2}^{\top}(x) \\
    ...... \\
    \tilde{a}_{j_m}^{\top}(x) 
\end{array}\! \! \! \right] ,   
\tilde{\mathbf{b}}_{ \mathcal{J}_0(\bar{z}, \bar{x})} (x) 
:=  \! \!  \left[\!\! \! \begin{array}{c}
  \tilde{b}_{j_1}(x) \\
  \tilde{b}_{j_2}(x) \\
  ...... \\
  \tilde{b}_{j_m}(x)  
\end{array}\!\! \!  \right]  , 
\end{aligned}
 \end{equation} 
where $j_1,j_2,...,j_m$ represent all indices in $\mathcal{J}_0(\bar{z}, \bar{x})$. In the above, variables $\bar{z}, \bar{x}$, appearing in the active set $\mathcal{J}_0(\cdot,\cdot)$, are introduced to make a difference from  the variable  $x$ in $\tilde{a}_{j}(x)$ and $\tilde{b}_{j}(x)$. Then the directional derivative of $ h_{r}(x) $ along the system can be obtained through \emph{Lemma \ref{theorem:PCP2}}. 

\begin{proposition}\label{prop:dir_der_h}
The directional derivative of $ h_{r}(x)$ in \eqref{eq:hxmin} at $x$ along $\dot{x}$ is 
\begin{align}
  h'_{r}(x ; \dot{x})  
  = \!\! \max _{z \in S(x)} \min_{\mu  \in M(z,x)} \!\!  
  -  \partial_{x} L(\mu, z, x) \dot{x}    , \label{eq:hxotxmin}
\end{align}
 where  the solution map $ S(x)$ is defined as  in \eqref{eq:Sxargmin}, 
 \begin{align}
  \partial_{x} L(\mu, z, x)    
  &= \!\sum_{j\in \mathcal{J}_0({z}, {x})}   \mu_{j} 
  \left( z^{\top}   \frac{\partial \tilde{ {a}}_{j}}{\partial x } (x)   -  
  \frac{\partial  \tilde{b}_{j}}{\partial x }(x)  \right),   \label{eq:partialxLmuzx}
  \end{align}
 and the set of multipliers  
\begin{equation}\label{eq:Mzxnuqt}
\begin{aligned}
M(z,x) 
= &\left\{\mu \mid  \tilde{\mathbf{A}}^{\top}_{ \mathcal{J}_0( {z}, {x})}( x) \mu =  - q , ~ \mu \geq 0  \right\}  ,   
\end{aligned} 
\end{equation}
is bounded (compact). 
\end{proposition}

\begin{proof}
 This proposition is the corollary of Lemma \ref{theorem:PCP2} with the boundness of set $M(z,x)$ to be proven. 
   To prove this,  first define 
$$
v_r(x):=  \min_{z  \in F_P(x)}  ~ q^{\top} z ,
$$ 
then $h_{r}(x)$ can be expressed as $h_{r}(x)= - v_r(x) - \epsilon_{0}$.

According to \eqref{eq:Lxtmu}, the local Lagrangian function  near $\bar{x},\bar{z} $ is defined as:
$$\begin{aligned}
L(\mu, z, x)  =   q^{\top} z +  \mu^{\top}   (   \tilde{\mathbf{A}}_{ \mathcal{J}_0(\bar{z}, \bar{x})}( x) z 
    - \tilde{\mathbf{b}}_{ \mathcal{J}_0(\bar{z}, \bar{x})}(x)   ), 
\end{aligned}$$
where $\mu \in \mathbb{R}^{  | \mathcal{J}_0({z},x)|}$, and  $\tilde{\mathbf{A}}_{ \mathcal{J}_0}, \tilde{\mathbf{b}}_{\mathcal{J}_0}$ are defined in \eqref{eq:J0zleft2}. 
Then,  the partial derivative of $L(\mu, z, x) $ w.r.t.  vectors  $z$ and $x$ are  respectively given by:    
$$\begin{aligned}
    \partial_{z} L(\mu, z, x) &= q^{\top} + \mu^{\top}  \tilde{\mathbf{A}}_{ \mathcal{J}_0( \bar{z},  \bar{x})}( x), \\
    \partial_{x} L(\mu, z, x) &= \mu^{\top} \left(  \frac{\partial  }{\partial x} 
    \left(  \tilde{\mathbf{A}}_{ \mathcal{J}_0(\bar{z}, \bar{x})}(x)z   \right) - \frac{\partial }{\partial x} \tilde{\mathbf{b}}_{ \mathcal{J}_0(\bar{z}, \bar{x})}(x)  \right) \nonumber  \\
    &= \sum_{j\in \mathcal{J}_0(\bar{z}, \bar{x})}   \mu_{j} 
    \left( z^{\top}   \frac{\partial \tilde{ {a}}_{j}}{\partial x } (x) - \frac{\partial  \tilde{{b}}_{j}}{\partial x }(x)  \right),   
    \end{aligned}$$
and $M(z,x):=  \left\{\mu  \mid  \partial_{z} L(\mu, z, x) = 0^{\top}, ~ \mu \geq 0    \right\}$.  
In the above, the partial derivative of $L(\mu, z, x) $ w.r.t. variable $z,x$ at the point $\bar{z},\bar{x}$ are calculated. By further choosing $\bar{z}=z$ and $\bar{x}=x$, the expressions in  \eqref{eq:partialxLmuzx} and \eqref{eq:Mzxnuqt} can be obtained. 
  
Since MFCQ condition holds and $F_{P}(x)$ is compact as proved in {Proposition} \ref{theo:pLPhx}, then the conditions of \emph{Lemma \ref{theorem:PCP2}} hold, from which one has 
\begin{align*}
  h'_{r}(x ; \dot{x})  &=- v'_r(x ; \dot{x})  
   = - \min _{z \in S(x)} \max _{\mu  \in M(z,x)} \partial_{x} L(\mu, z, x) \dot{x} \nonumber \\
    &=\max _{z \in S(x)} \min_{\mu  \in M(z,x)} - \partial_{x} L(\mu, z, x) \dot{x} .  
  \end{align*}
Hence, the expression in  \eqref{eq:hxotxmin} is obtained. 
    
The set $M(z,x)$ in \eqref{eq:Mzxnuqt} can be further rewritten as: 
   $$ \begin{aligned}
   M(z,x) 
   & =  \left\{\mu
   \mid  \sum_{j\in  \mathcal{J}_0( {z}, {x}) } \tilde{a}_j(x)   \mu_{j}  = - q , ~ \mu \geq 0  \right\} \\
   &=  M_{1}(z,x)   \cap  M_{2}(z,x) , 
   \end{aligned} $$
  where 
  $$
  \begin{aligned}
  M_{1}(z,x) &:=  \left\{\mu  \mid
      \sum_{j\in  \mathcal{J}_0( {z}, {x})} a_j(x)   \mu_{j} = 0 \right\} ,\\
  M_{2}(z,x) &:=  \left\{\mu  \mid 
      \sum_{j\in  \mathcal{J}_0( {z}, {x})} \| a_j(x)  \|  \mu_{j}=1 , ~ \mu \geq 0 \right\}.    
  \end{aligned} 
  $$
  From the structure of $ M_{2}(z,x)$, since $\| a_j(x)  \| >0 $ from Assumption \ref{ass:1}, 
   $ M_{2}(z,x)$ is a bounded cone \cite{boyd2004convex}, hence its subset $ M(z,x)$ is bounded naturally.  
\end{proof}

Based on {Proposition} \ref{prop:dir_der_h} and Definition \ref{def:ncbf}, the NCBF condition  to  guaranteeing $h_r(x(t))\geq 0, ~\forall t $ along the system \eqref{eq:N-affine-systems2} 
can be briefly expressed  as follows: 
\begin{align}
\! \max_{ \substack{ z \in S(x)  \\ u \in \Psi(x)   }  }  \min_{\mu \in M(z, x)}   \!  \! 
- \partial_{x} L(\mu, z, x)   \left(f +g u\right) \geq - \alpha ( h_{r} ). 
\label{eq:hrxfgu2}
\end{align}

It is worth mentioning that condition \eqref{eq:hrxfgu2} is more complex than the regular NCBF condition \eqref{eq:uinUminpart}. Due to the  (double-layer) max-min optimization form, it is not obvious to calculate the generalized gradient of $h_r(x)$. To simplify \eqref{eq:hrxfgu2} into a more practical form that is convenient to synthesize control laws, a sufficient condition is established. 
  
\begin{theorem}\label{theo:NCBFforPIM}
The BF candidate $h_{r}(\cdot)$  is further a NCBF on $\mathcal{S}$, i.e., the NCBF condition in \eqref{eq:cbf condition}  holds for $h_{r}(\cdot)$, if  there exist $\alpha , \bar{\epsilon}>0$,  for all $x\in  \mathcal{D}_{\bar{\epsilon}}:=\{x\mid h_r(x) \geq - \bar{\epsilon}\}$,  $   \exists u\in \Psi(x)$, such that
\begin{equation} \label{eq:partxL}
\min_{\mu \in M(z^*(x), x)}- \partial_{x} L(\mu,z^*(x),x) \left(f +g u\right) \geq - \alpha h_{r},  
\end{equation}
where $z^*(x)=\operatorname{col}{(c^*(x),r^*(x))}$ is any point chosen from the set $S(x)$ and the expression of $\partial_{x} L(\cdot) $ is given by \eqref{eq:partialxLmuzx}.
\end{theorem}

\begin{proof}
Based on \eqref{eq:hxotxmin} and \eqref{eq:partxL}, one has 
$$\begin{aligned}
&h'_{r}(x ; f(x)+g(x)u )\\
&=\max_{z \in S(x)} \min_{\mu \in M(z, x)}  
- \partial_{x} L(\mu, z, x) \left(f +g u\right)  \\
&\geq
  ~ \min_{\mu \in M(z^*(x), x)}  - 
      \partial_{x} L(\mu, z^*(x), x) \left(f +g u\right) 
\geq - \alpha h_{r}, 
\end{aligned}$$
which indicates  the NCBF condition \eqref{eq:hrxfgu2} holds.
\end{proof}

\begin{remark}
As far as the authors know, there are no existing works on how to simplify the kind of NCBF conditions in \eqref{eq:hrxfgu2}. Aimed at this, the condition \eqref{eq:partxL} is established for input $u$, which is similar to yet different from the regular NCBF condition \eqref{eq:uinUminpart}. The prerequisite of applying \eqref{eq:partxL} is selecting a point $z^*(x)$ from $S(x)$, such as through the existing open-sourced software packages about Chebyshev ball or through the following QP:
$ z^*(x)= \argmin_{z\in S(x)}  ~z^{\top}z$. 
How to choose a  non-conservative and continuous single-valued map $z^*(x)$ from the set-valued solution map $S(x)$ or its approximation is left to study in the future. $\hfill\square$
\end{remark}

\subsection{Synthesizing a NCBF-based Safety Filter}

Based on Theorem \ref{theo:NCBFforPIM}, this section attempts to synthesize a NCBF-based safety filter to solve PVM problems. 

Given the chosen point $z^*(x) \in S(x)$, recalling the expression  in  \eqref{eq:Mzxnuqt}, the dual  set $M(z^*(x),x)$ is  a bounded polytope in $\mathbb{R}^{|\mathcal{J}_0(z^*(x),x)|}$. 
From the famous Minkowski-Weyl’s Theorem \cite{weibel2007minkowski} that every polytope can be represented in two equivalent ways (i.e., halfspace- and  vertex-  representations), 
there exists a finite point set $V_{M}(x)=\left\{ {\mu}^{k}( x)  \mid k\in \mathcal{I}_{V_{M} }(x)  \right\}$, called vertexes\footnote{Till now,  it is convenient to obtain the vertex rep. $V_M(x)$ from half-space rep. of $M(z^*(x),x)$ in \eqref{eq:Mzxnuqt} automatically, such as the open-source toolbox $\mathrm{pypoman}$ \cite{caron2018pypoman}.}, such that  
\begin{align}\label{eq:Mzxx}
  M(z^*(x),x) = \operatorname{co} \left( V_{M}(x) \right),
\end{align} 
where $\operatorname{co} (\cdot) $ denotes the convex hull operator of a set, and $\mathcal{I}_{V_{M}}(x)$ is the index set  of ${V_{M}}(x)$.

Based on \eqref{eq:Mzxx}, a direct corollary of  Theorem \ref{theo:NCBFforPIM} is shown.
 \begin{corollary}\label{coro:NCBFforPIM}
  The BF candidate $h_{r}(\cdot)$  is further a NCBF on $\mathcal{S}$, 
     if  there exist $\alpha , \bar{\epsilon}>0$,  for all $x\in  \mathcal{D}_{\bar{\epsilon}}:=\{x\mid h_r(x) \geq - \bar{\epsilon}\}$, $\exists u\in \Psi(x)$, such that 
    \begin{equation} \label{eq:partxL2}
      \Gamma(\mu^{k}(x),x, u )
      \geq - \alpha  h_{r}(x)  , ~  \forall k\in \mathcal{I}_{V_{M} }(x),  
    \end{equation}
    where $\Gamma(\mu,x, u ) := -\partial_{x} L(\mu, z^*(x), x)  \left(f +g u\right)$ for any $\mu$. 
\end{corollary}

\begin{proof}
From the expression of $\partial_{x} L(\cdot)  $ in \eqref{eq:partialxLmuzx}, one can know that the scalar function $ \partial_{x} L(\mu, z^*, x) \left(f +g u\right)  $ is a linear function of $\mu, z^*,u$  respectively, hence the minimization  optimization about $\mu$ in \eqref{eq:partxL} is a LP at every $x$. Recalling the fact in \cite[Sec 2.5, Theorem and Corollary 2]{luenberger1984linear} that the minimum value of a LP is obtained at the vertexes  of the polytope, then 
$$\begin{aligned} 
\min_{\mu \in M(z^*(x), x)} \Gamma(\mu,x, u) 
&= \min_{{k} \in \mathcal{I}_{V_{M}}(x)}  \Gamma(\mu^{k},x, u ),
\end{aligned}
$$ which can be further combined with \eqref{eq:partxL} to verify \eqref{eq:partxL2}.
\end{proof}

Based on {Corollary} \ref{coro:NCBFforPIM}, a QP-based safety filter can be further built to solve Polytope Volume Monitoring:
\begin{subequations} \label{eq:sf-PVM}
  \begin{align}
    &  \bar{u}^*(x)  
   =  \argmin_{u, \delta}   ~  (u - u_0)^{\top}Q(u - u_0)
  + \gamma \delta^2 ,     \\
  & \text{s.t. }   \Gamma(\mu^{k}(x),x, u) \geq  -\alpha  h_{r}(x)  - \delta, ~ \forall k\in \mathcal{I}_{V_{M} }(x), 
  \label{eq:gammamuk} \\
  &\quad\quad  u\in\Psi(x),  ~ \delta\geq 0, \label{eq:upsix}
  \end{align}
\end{subequations}
where $\bar{u}^*(x):=\operatorname{col}(u^*(x) ,\delta^*(x))$, $u_{0}$ is any given continuous nominal control law, and the slack variable $ \delta $ is introduced to guarantee the feasibility of such QP. 
$Q\in \mathbb{R}^{m\times m},  \gamma>0$ are the (positive-definite) weighting matrix and scalar respectively, where the penalty scalar $\gamma$ is much larger than the maximal eigenvalue of $Q$.
The structure of the proposed safety filter \eqref{eq:sf-PVM} is summarized in Fig. \ref{fig:Figure-problem-solution} (b).

Till now,  the final theorem is shown to conclude this paper. 
\begin{theorem}\label{the:final}
  If for the designed control law $\bar{u}^*(x)= \operatorname{col}(u^*(x) ,\delta^*(x))$ in  \eqref{eq:sf-PVM}, it holds at every $x\in \mathcal{S}=\{ x\mid h_r(x) \geq 0  \}$ that  
  $ {u}^*(x)$  is  locally bounded and measurable, and the slack variable $\delta^*(x)$ is bounded by 
  \begin{align}
   \delta^*(x) \leq  \alpha  \epsilon_{0} .   \label{eq:deltastarx}
  \end{align} 
  Then the set $\mathcal{S} $ is forward invariant. 
  As a consequence,  for any trajectory $x(t)$ starting from $x_0 \in \mathcal{S}$, it holds that
  $$ 
  x(t)\in \mathcal{A}, ~ u^*(x(t))\in \Psi(x(t)), \quad \forall t\geq 0,
  $$ 
  i.e.,  the output set  is always nonempty and input constraint is never violated. 
\end{theorem}  

\begin{proof}
 Since  Eq. \eqref{eq:gammamuk}  denotes the  soft constraints, the feasibility of the QP-based safety filter is guaranteed whenever $\Psi(x)\neq \emptyset$. The solution of closed-loop system  is well-defined by the assumption that ${u}^*(x(t))$ is be measurable and  locally bounded \cite{glotfelter2020nonsmooth}. 
Meanwhile, the input constraint is never violated due to the hard constraints in \eqref{eq:upsix}.

Under the designed safety filter \eqref{eq:sf-PVM}, one has
\begin{equation}\label{eq:h_rdot}
  \begin{aligned}
(r^{*})'(x(t) ; \dot{x}(t) ) &=   h'_{r}(x(t) ; \dot{x}(t) ) \\
&  \geq -\alpha h_r(x(t)) - \delta^*(x(t)) \\
&=  -\alpha r^*(x(t)) + \alpha \epsilon_{0}   - \delta^*(x(t)).      
  \end{aligned}
\end{equation}
  When the inequality  \eqref{eq:deltastarx} holds, 
\eqref{eq:h_rdot} will be reduced to 
$(r^{*})'(x(t) ; \dot{x}(t) )  \geq  -\alpha r^*(x(t)) $. 
By referring to  Lemma \ref{lem:NCBF},  the set $\mathcal{S} $ is forward invariant. 
Recalling the relationship in \eqref{eq:rxtgeqe}, one has $x(t) \in \mathcal{A} = \{ x \mid \operatorname{Int} (\Phi(x)) \neq \emptyset \}, ~\forall t\geq0$. 
\end{proof}

\begin{remark}\label{remark:final}
In {Theorem} \ref{the:final},  
the condition $\delta^*(x) \leq \alpha  \epsilon_0$ can be guaranteed by choosing sufficiently large $\alpha \epsilon_0$. However,  large gain $\alpha$ leads to an agrressive behaviour and too large $\epsilon_0$ may cause the conservatism.  
How to further modify the proposed safety filter to a continuous one \cite{ong2023nonsmooth, isaly2024feasibility} or even  a smooth one \cite{cohen2023characterizing} is  left for the future study.  $\hfill\square$
\end{remark}

\section{Numerical Simulation Example}\label{sec:numerical-simu}
\label{sec:simu}
  
In this section, a case study of feasible space monitoring of multiple CBFs  as in Example \ref{exa1},  is carried out to evaluate the performance of the proposed safety filter \eqref{eq:sf-PVM}.  

\emph{1) Simulation Setting:} Consider the scenario in Fig. \ref{fig:Simuresults_comparisons} for a vehicle to perform a reach-avoid task under multiple obstacles and input constraints, in which  feasible space monitoring of multiple CBFs is required. The vehicle is modeled as a nonholonomic dynamic unicycle: 
 $$\dot{p}_x=v \cos \theta,  ~\dot{p}_y=v \sin \theta,  ~\dot{v}=a,  ~\dot{\theta}=\omega.$$ 
 Variables $x=[{p}_x,{p}_y,v, \theta]^{\top} $ are states, and  $u=[a,\omega]^{\top}$ are control inputs, restricted in $ \mathcal{U}=[-2,2]\times[-2,2] $. The reference controller in \cite{parwana2023feasible}, with parameters $k_p, k_\theta, k_v$,  is applied  as the nominal control $u_0$ to drive the vehicle from the initial point to the goal.  The barrier function candidate is designed as 
 $$h_i(x)= \|[{p}_x, {p}_y]^{\top} - [{o}_{xi}, {o}_{yi}]^{\top} \|^2  - R^2,$$
where $[{o}_{xi}, {o}_{yi}]^{\top}$ is the center of the $i$-th obstacle and $R$ is the pre-defined safety distance. The positions of obstacles, initial and goal points can be seen from Fig. \ref{fig:Simuresults_comparisons}.  The HOCBF \cite{tan2021high} (with order $2$ and  parameters $\alpha_1=10.0,\alpha_2= 6.0$) is utilized to transform the position constraints $h_i(x)\geq0$ to the affine constraints on input $u$, denoted as $u\in \mathcal{U}_c(x)$. According to Example \ref{exa1},  $\Phi(x)= \Psi(x)= \mathcal{U} \cap \mathcal{U}_{c}(x(t))$, the proposed safety filter \eqref{eq:sf-PVM} can be applied to solve the PVM problem here. 
  
To show the effectiveness of \eqref{eq:sf-PVM}  in improving the feasibility, it is compared with the usual safety filter that does not have the capability of feasible space monitoring: 
 \begin{align}
  u^*(x)  =  &  \argmin_{u\in\Psi(x) }   ~  (u - u_0)^{\top}Q(u - u_0).  \label{eq:sf-PVM-2}
   \end{align}
In the comparison, the same nominal control $u_0$ is applied with $k_p=0.5, k_\theta=1$, and
the same matrix $Q=[10,0;0,1]$ is applied in \eqref{eq:sf-PVM} and \eqref{eq:sf-PVM-2}. Other parameters only appeared  in \eqref{eq:sf-PVM} are set as  $\epsilon_0=0.6$, $\alpha= 15$ and $\gamma=500$.

\emph{2) Simulation Results:} The numerical simulation is conducted on a laptop with an Intel Core i7-10750H CPU and 16 GB RAM. A Python simulation environment is built, similar to the one in \cite{parwana2023feasible}. 
In \cite{parwana2023feasible}, the gradient of implicit NBF candidate $r^*(x)$ is approximated by some heuristic gradient value.
 In this paper, the in-depth analysis is performed to design an alternative one \eqref{eq:sf-PVM} without the need of the gradient of $h_r(x)$. 

\begin{figure*}[!htp]
  \centering
  \subfloat[The trajectories in the $XY$ plane. ]{\includegraphics[width=0.33\textwidth]{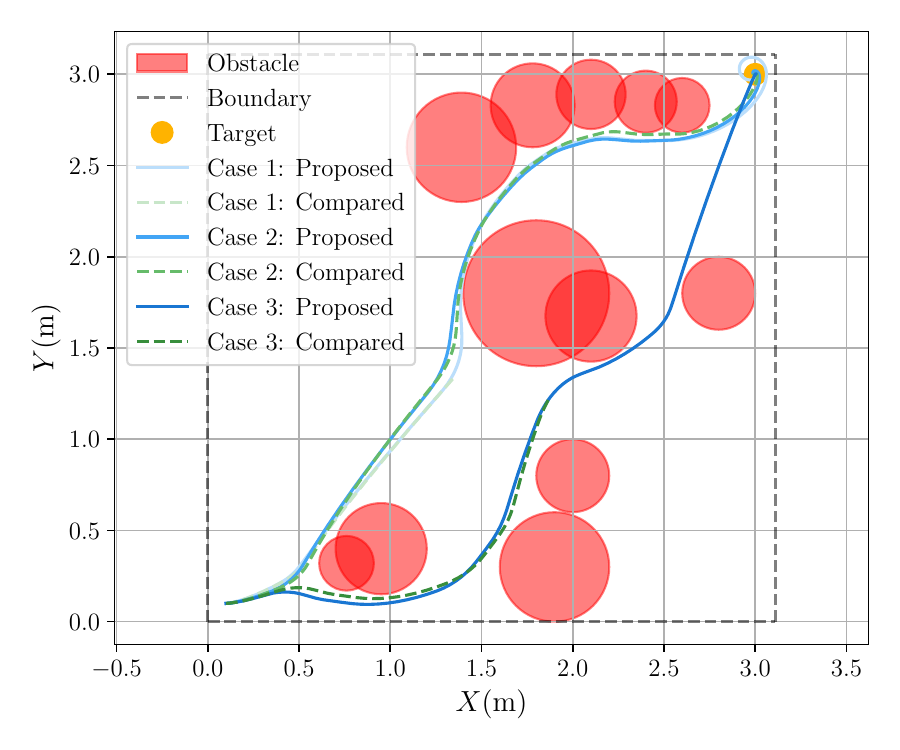}}
  \subfloat[The time response of $r^*(x(t))$.
  ]{\includegraphics[width=0.33\textwidth]{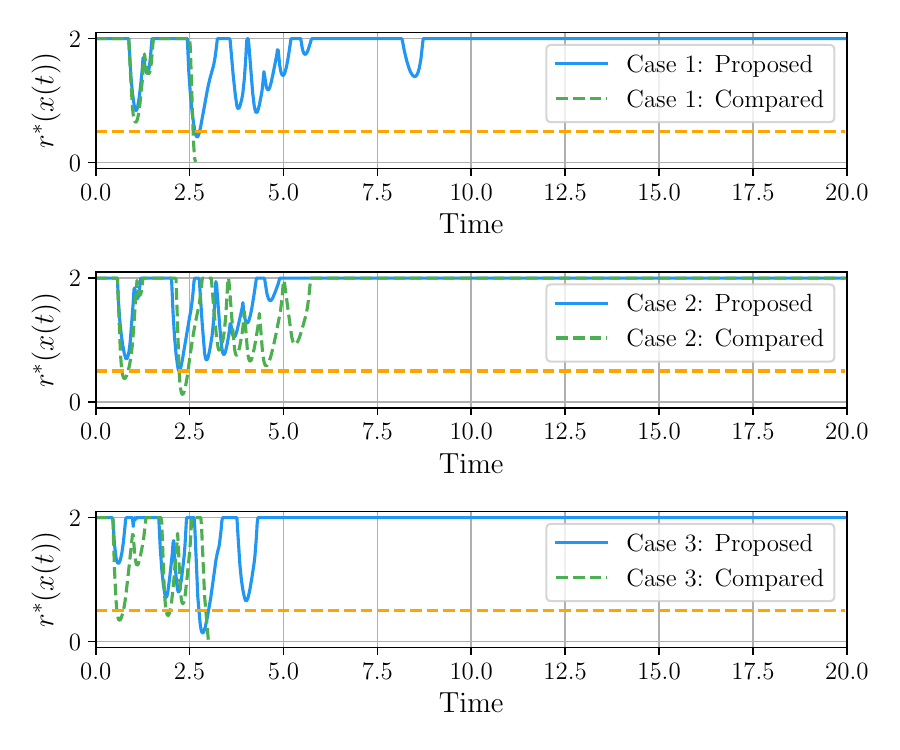}}
  \subfloat[Time responses of  $\delta^*(x(t))$ and $u^*(x(t))$.]{\includegraphics[width=0.33\textwidth]{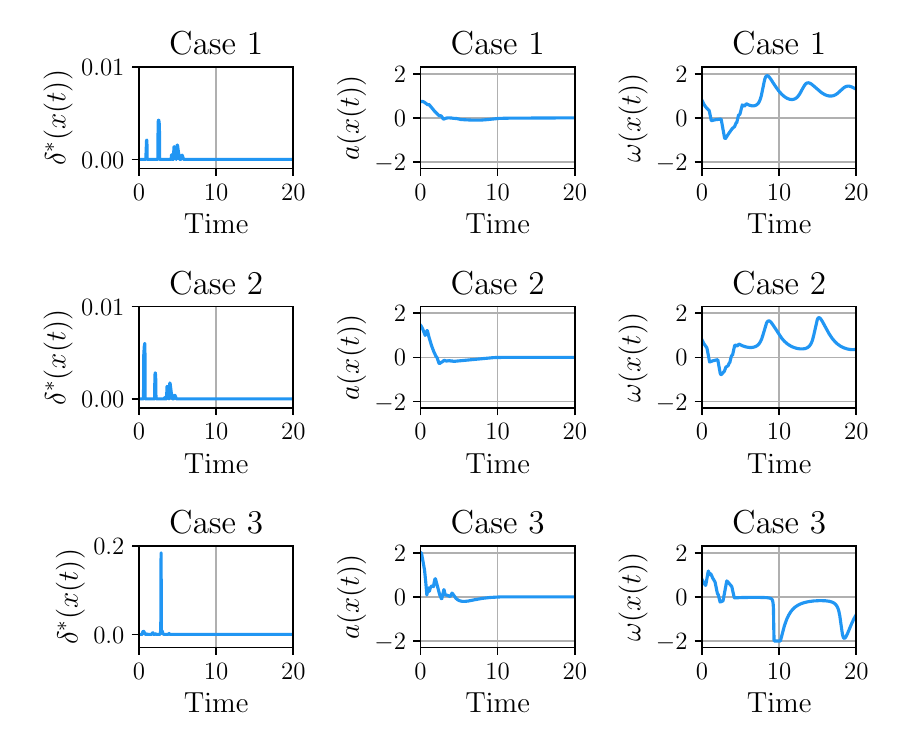}}
  \caption{Simulation results in different scenarios: Case 1 ($k_v=0.5$), Case 2 ($k_v=1.0$), and Case 3 ($k_v=2.0$). Subplots  (a), (b) illustrate the comparison between the proposed safety filter \eqref{eq:sf-PVM},  using blue solid line, and the usual one \eqref{eq:sf-PVM-2} described by the green dashed line. The orange line indicates the value of user-defined threshold value $ \epsilon_0=0.6 $. Subplot (c) shows the boundedness of slack variable $\delta^*(x(t))$ and the continuity of inputs $u^*(x(t))=[a(x(t)),\omega(x(t))]^{\top}$ under the proposed safety filter \eqref{eq:sf-PVM}. }
  \label{fig:Simuresults_comparisons}
    \vspace{-0.6cm}
\end{figure*}

To evaluate the performance sufficiently, the nominal control law $u_0$ with three different values of the gain $k_v$ are introduced in the simulation so as to produce different  moving behavior, named as Case 1 ($k_v=0.5$), Case 2 ($k_v=1.0$), and Case 3 ($k_v=2.0$), 
 as seen in Fig. \ref{fig:Simuresults_comparisons}. 
 
Fig. \ref{fig:Simuresults_comparisons}(a) shows the trajectories of the vehicle in the $XY$ plane. It can be observed that 
under the proposed safety filter \eqref{eq:sf-PVM}, the vehicle successfully avoids the obstacles and reaches the goal in all cases.  Under the compared method \eqref{eq:sf-PVM-2}, in Case 2, the vehicle can successfully avoid the obstacles and reaches the goal, while in Case 1 and Case 3, the vehicle would get stuck in the obstacles. 
Fig. \ref{fig:Simuresults_comparisons}(b) illustrates the time response of $r^*(x(t))$,  which demonstrates the range of the feasible  space of these two different safety filters. From it, one can observe that $r^*(x(t))$ is  always  greater than zero and greater than $ \epsilon_0=0.6 $ in most cases, which shows that the proposed method can maintain feasibility.
As a comparison,  the traditional one \eqref{eq:sf-PVM-2} becomes infeasible in Case 1 and Case 3, and its feasible  space is much smaller than the proposed one's in Case 2.  
In Fig. \ref{fig:Simuresults_comparisons}(c),  the time responses of slack variable $\delta^*(x(t))$ and inputs  $u^*(x(t))=[a(x(t)),\omega(x(t))]^{\top}$ under the proposed safety filter \eqref{eq:sf-PVM} in the above three cases are shown. One can observe that in all cases, 
$\delta^*(x(t))$ is equal to $0$ approximately and much lower than the threshold $ \alpha  \epsilon_0=9$. 
Besides, the inputs $u^*(x(t))$ are continuous in all cases. 
 These observations confirm  the reasonability of the theoretical assumptions, analysis and  conclusions in Theorem \ref{the:final} and Remark \ref{remark:final}. 
 
\section{Conclusion}
\label{sec:conclusion}

This paper formulates a class of control problems, called polytope volume monitoring  (PVM), for which a parametric linear program-based control barrier function approach is proposed.
This is achieved by establishing a bridge between the existing nonsmooth CBF and parametric optimization theory, where some useful lemmas about the continuity and directional differentiability of
 parametric programs are utilized to build nonsmooth CBF conditions, based on which a slacked QP-based safety filter is proposed to tackle the PVM problem step-by-step. The biggest limitation of this work is that the continuity of the proposed safety filter is not discussed formally.
Future works involve designing continuous and smooth safety filters to tackle PVM problems, extending the proposed analysis and synthesis framework to  other relevant problems and so on. 
 
\appendix

\setcounter{equation}{0}
\renewcommand{\theequation}{A.\arabic{equation}}
\renewcommand{\thelemma}{A.\arabic{lemma}}
\renewcommand{\thedefinition}{A.\arabic{definition}}
\setcounter{lemma}{0}
\setcounter{definition}{0}

\section{Appendix}
\subsection{Nonsmooth Parametric Optimization}\label{sub:pcp}

Let $\mathbb{T} \subset \mathbb{R}^n$ denote some open parameter set.  For all $\bar{x} \in \mathbb{T}$, consider  the convex parametric program: 
\begin{align}
 v(\bar{x}):= \min_{\bar{z}} &~ \tilde{f}  (\bar{z}, \bar{x}),  \label{eq:vtminxfxt} \\
\text{ s.t. } &\bar{z} \in {F}(\bar{x}):=\left\{\bar{z} \in \mathbb{R}^{n_z} \mid   \tilde{g}_j(\bar{z}, \bar{x}) \leq 0, j \in \mathcal{J}\right\}, \nonumber
\end{align}
with $\mathcal{J}:=\{1, 2, \ldots, N\}$. For convenience, let  the set of (global) minimizers according to \eqref{eq:vtminxfxt} be   denoted as 
\begin{equation}\label{eq:Sbartxin}
  S( \bar{x})  :=\left\{\bar{z} \in F(\bar{x}) \mid \tilde{f}(\bar{z},  \bar{x})=v(\bar{x})\right\}.  
\end{equation}
For $\bar{x} \in \mathbb{T}$ and feasible $\bar{z} \in F(\bar{x})$, the active index set is defined as $\mathcal{J}_0(\bar{z}, \bar{x}):=\left\{j \in \mathcal{J} \mid \tilde{g}_j(\bar{z}, \bar{x})=0\right\}$. 
Throughout this section, $\tilde{f},  \tilde{g}_j, j \in \mathcal{J}$, are continuously differentiable on $\mathbb{R}^{n_z} \times \mathbb{T}$, notation $\tilde{f},  \tilde{g}_j \in \mathcal{C}^1$. The following Constraint Qualification conditions can be found in \cite[(1.3) and (1.4)]{still2018lectures} and \cite[(11) and (12)]{solodov2010constraint}.
  
\begin{definition}\label{def:LICQ-MFCQ}
  There are two constraint qualification conditions, respectively defined as:

  1) The 
   Mangasarian Fromovitz
    Constraint Qualification (MFCQ) is said to hold at $(\bar{z}, \bar{x})$ for $\bar{z} \in F(\bar{x})$ if
     there exists  $\xi$ such that $
   \partial_z \tilde{g}_j(\bar{z}, \bar{x}) \xi<0,  j \in \mathcal{J}_0(\bar{z}, \bar{x})$. 
  
  2) Linear Independence Constraint Qualification (LICQ)  holds at $(\bar{z}, \bar{x})$   if  vectors    $ \partial_z \tilde{g}_j(\bar{z}, \bar{x}), j \in \mathcal{J}_0(\bar{z}, \bar{x})$ are linearly independent.  
\end{definition}
  
The continuity and directional differentiability  properties of the value function of parametric programs  with only inequalities  have been studied in \cite{still2018lectures}. 
 
\begin{lemma}\cite[Lemma 6.2]{still2018lectures}\label{theorem:PCP}
   Consider the parametric convex program  in  \eqref{eq:vtminxfxt} and assume there exists a neighborhood $B_{\varepsilon_0}(\bar{x}), \varepsilon_0>0$, such that $\cup_{x \in B_{\varepsilon_0}(\bar{x})} F(x) \subset C_0$, with $C_0$ compact.
  If the MFCQ condition holds at $(\bar{z}, \bar{x})$ for any $\bar{z} \in S(\bar{x})$, then the value function $v(\cdot)$ is locally Lipschitz continuous at $\bar{x}$.
\end{lemma}
    
  Define the (local) Lagrangian function near $(\bar{z}, \bar{x})$ as:  
\begin{equation}\label{eq:Lxtmu}
  L(\mu, z, x )=\tilde{f}(z, x)+\sum_{j \in \mathcal{J}_0(\bar{z}, \bar{x})} \mu_j \tilde{g}_j(z, x),  
\end{equation}
where $ {\mu}:=\left({\mu}_j, j \in \mathcal{J}_0(\bar{z}, \bar{x})\right)\in \mathbb{R}^{\left| \mathcal{J}_0(\bar{z}, \bar{x})\right|}$, collecting all multiplers $\mu_j$ corresponding to active constraints in $ \mathcal{J}_0(\bar{z}, \bar{x})$. Then the partial derivatives of $L(\cdot)$  are respectively given by   
\begin{subequations}\label{eq:Lxtmut}
\begin{align}
\partial_x L(\mu, z, x) :=&\partial_x \tilde{f}(z, x) +\sum_{j \in \mathcal{J}_0(\bar{z}, \bar{x})}  \mu_j  \partial_x   \tilde{g}_j(z, x), \nonumber \\  
\partial_z L(\mu, z, x) :=&\partial_z \tilde{f}(z, x) +\sum_{j \in \mathcal{J}_0(\bar{z}, \bar{x})}  \mu_j  \partial_z \tilde{g}_j(z, x). \nonumber 
\end{align}
\end{subequations}
 
\begin{lemma} \cite[Theorem 7.3]{still2018lectures} \label{theorem:PCP2}
  Consider the parametric convex program  in  \eqref{eq:vtminxfxt}   and let $F(\bar{x})$ be compact. If 
   MFCQ  holds at  $(\bar{z}, \bar{x})$ satisfying $\bar{z} \in F(\bar{x})$,
  then $v(\cdot)$ is directional differentiable at $\bar{x}$ in any direction $d ,\|d\|=1$ with
  \begin{align}
  v'(\bar{x} ; d)=\min _{\bar{z} \in S(\bar{x})} 
   \max _{\mu  \in M(\bar{z}, \bar{x})} \partial_x L(\mu, \bar{z}, \bar{x} ) d , 
  \label{eq:nablavz}
  \end{align}
  where  the  minimizer set $S(\cdot)$ is defined in \eqref{eq:Sbartxin}, and the set of multipliers $M(\cdot)$ is defined as  
  \begin{align}
   M(\bar{z}, \bar{x}) &:=\left\{\mu \mid \partial_z L(\mu, \bar{z}, \bar{x} )=0^{\top},  \mu \geq 0 \right\}.  
   \label{eq:Mbarxbart}
  \end{align}
  \end{lemma}
   
  If LICQ holds, then $M(\bar{z}, \bar{x})$ will be a singleton, i.e., the multiplier ${\mu}$ will be uniquely determined \cite[EX. 2.6]{still2018lectures}.
The min-max expression of directional derivative in \eqref{eq:nablavz} has also been pointed out by  other literature 
  \cite[Theorem 2.4]{best1990stability}, \cite[Theorem 4.24]{bonnans2013perturbation}
  yet under different conditions.

\bibliographystyle{IEEEtran}
\bibliography{bio/bibliography.bib}
\end{document}